\documentclass{amsart}
\usepackage{amssymb}
\usepackage{amsfonts}
\usepackage{url}

\theoremstyle{plain}
\newtheorem{theorem}{Theorem}

\newtheorem{lemma}[theorem]{Lemma}
\newtheorem{proposition}[theorem]{Proposition}
\newtheorem{assumption}[theorem]{Assumption}

\theoremstyle{remark}

\newtheorem*{ack}{Acknowledgements}

\theoremstyle{definition}
\newtheorem{remark}[theorem]{Remark}
\newtheorem{definition}[theorem]{Definition}
\newtheorem{example}[theorem]{Example}

\numberwithin{equation}{section}
\numberwithin{theorem}{section}

\DeclareMathOperator{\codim}{codim}

\newcommand{\bM}{\overline{M}}
\newcommand{\en}{ext\nu}

\newcommand{\rintfrac}[2]{\genfrac{\lceil}{\rceil}{}{1}{#1}{#2}}

\newcommand{\Zol}{\.Zo\l{}\c{a}dek {}}

\begin{document}
\title[Deformations of singularities]{Deformations of singularities of plane curves. Topological approach.}
\author{Maciej Borodzik}
\address{Institute of Mathematics, University of Warsaw, ul. Banacha 2,
02-097 Warsaw, Poland}
\email{mcboro@mimuw.edu.pl}
\date{23 September 2009}
\thanks{Supported by Polish MNiSz Grant N N201 397937. The author is also supported by Foundation for Polish Science (FNP)}

\begin{abstract}
In this paper we use a knot invariant, namely the Tristram--Levine signature, to study deformations of singular points of
plane curves. We bound, in some cases, the difference between the $M$ number of the singularity of 
the central fiber and the sum of $M$ numbers of the generic fiber.
\end{abstract}

\maketitle
\section{Introduction}
A deformation of a plane curve singularity is, roughly speaking, a smooth family of germs of plane algebraic curves $\{C_s\}_{s\in D}$ (we
consider here only deformations over a disk in $\mathbb{C}$) 
such that $C_s\subset\mathbb{C}^2$ and a distinguished member, say $C_{0}$, has a singular point at $0\in\mathbb{C}^2$. 
The question we
address is the following: how are related to each other singular points of $C_{0}$ and of $C_{s}$ with $s$ sufficiently small?
This question, although already very difficult, becomes even more involved if we impose some topological 
constrains on the general members $C_s$.
For example, we can require all of them to be rational, which means that each $C_s$ is a union of immersed disks.

The rationality condition is justified for various reasons. For example, let us be given a flat family $C_s$ of 
projective curves in some 
surface $Z$ and this family specialises to 
a curve $C_0$ with the same geometric genus as $C_s$. Then, for each singular point $z\in C_0$, we can take
a sufficiently small ball $B$ around $z$ and the family $C_s\cap B$ provides a deformation of a singular point such that
all curves $C_s\cap B$ are rational.

To show a more specific example, we can take $C=C_{mn}$ to be a polynomial curve given in parametric form
by $C=\{(t^n,t^m),t\in\mathbb{C}\}$ with $n,m$ coprime, and assume $C'$ is also parametric $C'=\{(x(t),y(t)),t\in\mathbb{C}\}$
with $\deg x=n$, $\deg y=m$.
Then for $s\in\mathbb{C}\setminus\{0\}$, the
mapping $(s^{n}x(t/s),s^{m}y(t/s))$
parametrises a curve that is algebraically isomorphic to $C'$ and, for sufficiently small $s$, is very close to $C$. In other
words, every polynomial curve of bidegree $(n,m)$ specialises to $(t^n,t^m)$. In particular if a polynomial curve
of bidegree $(n,m)$ has some singularity, this singularity can be specialised to the quasi-homogeneous singularity $(t^n,t^m)$.
So, classifying parametric deformations encompasses the problem of finding possible singularities of a polynomial curve
of a given bidegree. 
The characterisation of possible singularities of polynomial curves is, in turn, a problem with applications beyond algebraic
geometry itself, for example in determining the order of weak focus of some ODE systems (see \cite{ChLy} and \cite[Section 5]{BZ}).

In \cite{Or,BZ} there was defined a new invariant of plane curve singularities, namely a codimension (or, as Orevkov calls it,
a rough $\bM$-number). It is, roughly speaking, the codimension of the (topological) equisingularity
stratum in the appropriate space of parametric singularities. A naive parameter counting argument suggests that
this invariant is 
upper-semicontinuous under parametric deformations. 
Yet proving this appears to be an extremely difficult task. On the one hand, the $\bM$ number
can be expressed by some intersection number of divisors in the resolution of singularity, but then
the blow--up diagram changes after a deformation in a way that we are still far from understand. 
In an algebraic approach, the geometric genus of nearby fibers is quite difficult to control.
On the other hand, the famous Hirano's example \cite{Hir} can be used to show, that a naive generalisation of this 
expected semicontinuity property fails if we allow the curves $C_s$
to have higher genera.

A possible rescue comes from a very unexpected place, namely from knot theory. It turns out that the $\bM$ number, or its more
subtle brother, the $M$ number, is very closely related to the integral of the Tristram--Levine signature of the knot
of the singularity (\cite{Bo2}). We say a knot, instead of a link, to emphasize that this relationship has been proved only
in the case of cuspidal singularities. On the other hand, we can apply methods from \cite{Bo} to study the changes of the
Tristram--Levine signature. Putting things together we obtain a bound for the difference between the sum of $M$-numbers
of singular points of a generic fiber and the sum of $M$-numbers of singular points of the central fiber, provided that the
curves have only cuspidal singularities or double points.

Alas, this result is not that strong as we could hope to prove. However, to the knowledge of the author, it is one of very few results on
that subject. Moreover the result can possibly be improved by applying different knot invariants than the Tristram--Levine signature.
  
The structure of the paper is the following. First we precise, what is a deformation (Section~\ref{Sub1}). 
Then we recall definitions of codimension (Section~\ref{codimension}). Section~\ref{tristram} is devoted to the application of the 
Tristram--Levine signature. We recall a definition of the Tristram--Levine signature and cite two results
from \cite{Bo} and \cite{Bo2}. This allows to provide the promised estimates in Section~\ref{estimates}.

\section{What is a deformation?}\label{Sub1}
By a \emph{deformation} of a plane curve singularity over a base space $(D,0)$, where $D\subset\mathbb{C}$ is an open disk, 
we mean a pair $(\mathcal{X},X_0)$, where $\mathcal{X}$ is
a germ of an algebraic surface (called a \emph{total space}) 
and $X_0\subset\mathcal{X}$ is a curve (called a \emph{central fiber}), together
with a flat morphism $\pi\colon (\mathcal{X},X_0)\to (D,0)$ and an algebraic map $F\colon \mathcal{X}\to \mathbb{C}^2$ such that 
$X_s=\pi^{-1}(s)$ is a (germ of an) algebraic curve and $F|_{X_s}$ is generically one to one on its image. 
\begin{remark}
If it does not lead to a confusion we shall identify $X_s$ with its image $F(X_s)\subset\mathbb{C}^2$.
\end{remark}
We put some additional technical conditions on the deformation, motivated by the Milnor fibration.
\begin{itemize}
\item[(D1)] $X_0$ is homeomorphic either to a disk or to a bunch of disks glued at $0\in\mathbb{C}^2$. In both
cases $X_0$ is smooth away from $0$;
\item[(D2)] There exists a ball $B=B(0,\delta)\subset\mathbb{C}^2$ such that $X_s\subset B$, $\partial X_s\subset\partial B$ and
$X_s$ is transverse to $\partial B$;
\item[(D3)] The intersection $X_0\cap\partial B$ is the link of the singularity of $X_0$ at $0$.
\end{itemize}
If we are given a deformation $(\tilde{\mathcal{X}},\tilde{X_0})$ not necessarily satisfying conditions (D1)--(D3),
we may chose $\delta$ so small
that $\tilde{X}_{0}$ is transverse to $\partial B(0,\delta)$ and (D1), (D3) are satisfied. Then we may shrink the 
base disk $D$, if necessary,
such that for all $s\in D$, $\tilde{X}_s$ is still transverse to $\partial B(0,\delta)$. If we define
$\mathcal{X}=F^{-1}(B(0,\delta))$ then this new deformation has already properties (D1)--(D3). 

From now on, a deformation will always mean a deformation of a plane curve singularity satisfying conditions (D1), (D2) and (D3).
\begin{definition}
The \emph{genus} $g$ of the deformation is the geometric genus (i.e. the topological genus of the normalisation)
of a generic fiber $X_s$. The deformation is \emph{rational} if
$g=0$, in which case all $X_s$ are sums of immersed disks. The deformation is \emph{unibranched} if $X_0$ is a disk. The deformation
is \emph{parametric} if it is both rational and unibranched.
\end{definition}
The intersection of $X_s$ with the ball $B$ from Property~(D2) is a link, which we shall denote $L_s$. As this intersection
is transverse for each $s\in D$, the isotopy type of $L_s$ does not depend on $s$.
\begin{definition}
The (isotopy class of the) link $L_s$ is called the \emph{link of the deformation}. It is denoted by $L_X$.
\end{definition}
\begin{remark}
Property~(D3) ensures that $L_X$ can be identified with the link of singularity of $F(X_0)$.
\end{remark}
\begin{lemma}
Let $(\mathcal{X},X_0)$ be a parametric deformation. Then, there exists such an $\varepsilon'<\varepsilon$ and a family of
holomorphic functions
\begin{align*}
x_s(t)&=a_0(s)+a_1(s)t+\dots\\
y_s(t)&=b_0(s)+b_1(s)t+\dots
\end{align*}
with $|s|<\varepsilon$ that $(x_s,y_s)$ locally parametrises $X_s$ and both $x_s$ and $y_s$ depend analytically on $s$.
\end{lemma}
\begin{proof}
The assumptions on the parametricity and transversality guarantee that the deformation is $\delta-$constant, hence equinormalisable
(see \cite[Section~2.6]{GLS}).
By assumptions, the normalisation of $\mathcal{X}$ is a product $D\times D'$, where $D'$ is a small disk. 
Let $\rho$ be a normalisation map.
The maps $x_s$ and $y_s$ are recovered by projecting the composition $F\circ\rho$ onto corresponding coordinates in $\mathbb{C}^2$.
\end{proof}

\section{Codimension}\label{codimension}
The codimension is a topological invariant of a plane curve singularity.
We recall here a definition from \cite{BZ}.
\begin{definition}
Let $\mathcal{T}$ be a topological type of a plane curve cuspidal singularity with 
multiplicity $m$. Let $\mathcal{H}$ be the space of polynomials 
in one variable. Consider the stratum  $\Sigma_C\subset \mathcal{H}$ consisting
of such polynomials $y$ that a singularity parametrised by
\[
t\to (t^m,y(t)) 
\]
defines a singularity at $0$ of type $\mathcal{T}$. Then the \emph{external codimension} of the singularity $\mathcal{T}$
is
\[
\en=\codim(\Sigma\subset\mathcal{H})+m-2.
\]
\end{definition}
The interpretation of the definition is the following. If we consider the space of pairs of polynomials $(x(t),y(t))$ of 
sufficiently high degree,
then the subset of those parametrising a curve with a singularity of type $\mathcal{T}$ forms a subspace of codimension 
$\en(\mathcal{T})$. In fact, there
are $m-1$ condition for the derivatives of $x$ to vanish at some point, $\codim(\Sigma\subset\mathcal{H})$ conditions for 
the polynomial $y$ (the degree of $y$
is assumed to be high enough so that these conditions are independent). The missing $-1$ comes from the fact that we do 
not require the singularity to be
at $t=0$, but we have here sort of freedom.

\begin{remark}
In \cite{BZ} the assumption that $m$ is the multiplicity is not required. If $m$ is not the multiplicity, then \eqref{eq:BZ}
below, does no longer hold.
\end{remark}
The above definition can be generalised to multibranched singularities. We refer to \cite{BZ} for detailed definitions. 

There exists also a construction of the $\bM$ number in a coordinate--free way. It can be done as follows.
Let $(C,0)$ be a germ of a plane curve singularity at $0$, not necessarily unibranched. Let $\pi:(U,E)\to(C,0)$ be the 
minimal resolution of this singularity, where $E=\sum E_i$
is the exceptional divisor with a reduced structure. Let $K$ be a (local) canonical divisor on $U$, which means that 
$K=\sum\alpha_iE_i$ and $(K+E_i)\cdot E_i=-2$
for exceptional curves $E_i$. Let $C'$ be the class of the strict transform of $C$, and $D=C'+E$.
\begin{definition}\label{roughM}
A \emph{rough $\bM$-number} of $(C,0)$ is the quantity
\[K\cdot (K+D).\]
\end{definition}
We have the following fact (see \cite{BZ})
\begin{equation}\label{eq:BZ}
\bM=\en.
\end{equation}
Orevkov \cite{Or} defines, besides a rough $\bM$-number, a fine $M$-number of a singularity. We should take
the Zariski--Fujita decomposition
\[K+D=H+N,\]
where $H$ is positive and $N$ nef. We have the following definition (see \cite[Definition~4.3]{BZ}).
\begin{definition}
The $M$ number of the singularity is equal to $\bM-N^2$.
\end{definition}

\noindent $N^2$ is always non-positive, so $\bM\le M$. 
For cuspidal singularities we have $N^2<-1/2$, while for an ordinary $d-$tuple point $N=0$.

Both $\bar M$ and $M$ numbers can be very effectively calculated from the Eisenbud--Neumann diagram. 
An algorithm can be found for
example in \cite{BZ}. We provide a simple, but important example.
\begin{example}\label{e:Mpq}
Let $p,q$ be coprime positive integers and consider the singularity $\{x^p-y^q=0\}$. Its $\bM$ number is equal to
$p+q-\rintfrac{p}{q}-\rintfrac{q}{p}-1$, while
\begin{equation}\label{eq:Mpq}
M=p+q-\frac{p}{q}-\frac{q}{p}-1.
\end{equation}
\end{example}
\begin{example}
Both $\bM$ and $M$ numbers of an ordinary double point are zero.
\end{example}
\section{Tristram--Levine signatures}\label{tristram}

Let $L$ be a link in $S^3$. Let $V$ be a Seifert matrix of $L$. Let $\zeta\in\mathbb{C}$, $|\zeta|=1$.
\begin{definition}
The \emph{Tristram--Levine signature} of $L$ is the signature $\sigma_L(\zeta)$ of the Hermitian form given by the matrix
\[(1-\zeta)V+(1-\bar{\zeta})V^T.\]
\end{definition} 
It is well-known that $\sigma_L$ is a link invariant. It is also easily computable for algebraic links.
\begin{example}\label{e:sigTpq}
Let us consider the singularity $\{x^p-y^q=0\}$ as in Example~\ref{e:Mpq} and let $T_{p,q}$ be its link (note, that this
is exactly the $(p,q)$-torus knot). Its Tristram--Levine signature can be computed as follows: consider a set
\[\Sigma=\left\{\frac{i}{p}+\frac{j}{q}\colon 1\le i\le p-1,\,\,1\le j\le q-1\right\}\subset(0,2).\]
Let $\zeta=e^{2\pi i x}$ with $x\in(0,1)$ and $x\not\in\Sigma$. Then
\[\sigma(\zeta)=-\#\Sigma\cap(x,x+1)+\#\Sigma\setminus(x,x+1).\]
Here $\#$ denote the cardinality of a finite set.
\end{example}
In general, $\sigma(\zeta)$ is a piecewise constant function with jumps only at the roots of the Alexander polynomial. Its values
are computable, yet they can not always be expressed by a nice, compact formula. However, the main feature we shall
use is that Tristram--Levine signatures behave well under knot cobordism. This behaviour was studied in \cite{Bo} in
the context of the plane algebraic curves. We use one result from this paper, that in our setting can be formulated as follows.

Assume $(\mathcal{X},X_0)$ is a deformation. Let $Y=X_s$ be a non-central fiber (i.e. $s\neq 0$). Assume that $z_1,\dots,z_N$
are the singular points of $Y$ and $L_1,\dots,L_N$ the corresponding links of singularities. Let, finally, $b_1(Y)$ denotes
the first Betti number of $Y$. Recall that $L_0$ is the link of the singularity $X_0$.
\begin{proposition}\label{p:Bo} For almost all $\zeta\in S^1$
\begin{equation}\label{eq:Bo}
|\sigma_{L_0}(\zeta)-\sum_{k=1}^N\sigma_{L_k}(\zeta)|\le b_1(Y).
\end{equation}
\end{proposition}
\begin{proof}
Let $x,y$ be the coordinates in $\mathbb{C}^2$. If the function $|x|^2+|y|^2$ is Morse on $Y$, then the
statement follows from \cite[Remark~6.8]{Bo} ($L_0$ in the present paper corresponds to $L_r$ in \cite{Bo}).
If the above function is not Morse, we can still find its subharmonic
perturbation which is sufficiently close to the original one in $B(0,\delta)$ and finish the proof in the way like above.
\end{proof}
Proposition~\ref{p:Bo} gives a strong obstruction for the singularities occurring in the perturbations. Yet the Tristram--Levine
signature function is difficult to handle as we have already seen in Example~\ref{e:sigTpq}. Fortunately, there
is a result of \cite{Bo2} that allows to draw some consequences from Proposition~\ref{p:Bo} in a ready-to-use form.
\begin{proposition}\label{p:Bo2} Let $C$ be a germ of a curve singular at $z_0$. Let $K$ be the corresponding
link of the singularity, $\mu$ and $M$
the Milnor and $M$ numbers of $C$. If $K$ is a \emph{knot} then
\begin{equation}\label{eq:Bo2}
0<-3\int_0^1\sigma(e^{2\pi i x})dx-M-\mu<\frac{2}{9}.
\end{equation}
\end{proposition}
Now we have all pieces to prove the main result.
\section{The main result}\label{estimates}
The setup in this section is the following. $\mathcal{X}$ is a deformation, $X_0$ the central fiber and $Y=X_s$ ($s\neq0$) some
generic fiber. We introduce the following notation:
\begin{itemize}
\item $\mu_0$ is the Milnor number of the singularity of $X_0$ and $M_0$ its $M$ number;
\item $g$ is the geometric genus of $Y$;
\item $z_1,\dots,z_N$ are singular points of $Y$, $L_1,\dots,L_N$ are corresponding links of singularities. Then 
$\mu_1,\dots,\mu_N$ (respectively $M_1$,\dots,$M_N$)
are Milnor numbers (resp. M--numbers) of the singular points;
\item $b_1$ is the first Betti number of $Y$.
\end{itemize}
We shall put a following additional assumption. It is dictated by the fact that we do not have the formula for the integral
of the Tristram--Levine signature for general algebraic links.
\begin{assumption}
There is $n\le N$ that $z_1,\dots,z_n$ are cuspidal and $z_{n+1},\dots,z_N$ are ordinary double points.
\end{assumption}
Let
\[R=N-n\]
be the number of the double points of $Y$. We have the following important result.
\begin{theorem}\label{mainthe}
In the above notation.
\begin{equation}\label{eq:sig}
\sum_{k=1}^nM_k-M_0< 8g+2R+\frac29.
\end{equation}
\end{theorem}
\begin{proof}
Let us observe that
\begin{align}
b_1(Y)&=2g+R\label{eq:betti}\\
\mu_0&=2g+R+\sum_{k=1}^N\mu_k=2g+2R+\sum_{k=1}^n\mu_n.\label{eq:genus}
\end{align}
The equality~\eqref{eq:genus} is exactly the genus formula. 
It can be proved by comparing the Euler characteristics of smoothings of $X_0$ and $Y$ (they must agree).
Since the signature of a link of a double point is exactly $-1$ we deduce from Proposition~\ref{p:Bo} that for almost all $\zeta$
\begin{equation}\label{eq:sig2}
\sum_{k=1}^n(-\sigma_{L_k}(\zeta))-(-\sigma_{L_0}(\zeta))\le 2g.
\end{equation}
The signs in \eqref{eq:sig2} are written in this way on purpose. 
Now we integrate the inequality \eqref{eq:sig2}. Using \eqref{eq:Bo2} we get
\[\sum_{k=1}^n(\mu_k+M_k)-\mu_0-M_0< 6g+\frac29.\]
Applying \eqref{eq:genus} finishes the proof.
\end{proof}
We see that in this approach, the control of the genus is vital. In particular we can have the following result.
\begin{proposition}[BMY like estimate]\label{BMYlike}
Let $C$ be a curve in $\mathbb{C}^2$ given in parametric form by
\[C=\{(x(t),y(t)), t\in\mathbb{C}\},\]
where $x$ and $y$ are polynomials of degree $p$ and $q$ respectively. Assume that $p$ and $q$ are coprime and 
$C$ has cuspidal singularities $z_1,\dots,z_k$
with $M$--numbers $M_1,\dots,M_k$ and, besides, $C$ has precisely $R$ double points. Then
\[\sum_{k=1}^n M_k< p+q-\frac{p}{q}-\frac{q}{p}-\frac79+2R.\] 
\end{proposition}
\begin{proof}
Consider a family of curves
\[C_s=\{(s^p x(s^{-1} t),s^q y(s^{-1} t),t\in\mathbb{C}\},\]
where $s$ is in the unit disk in $\mathbb{C}$.
For $s\neq 0$ all these curves are isomorphic, while for $s=0$ we have a homogeneous curve $(t^p,t^q)$. 
Let $B$ be a sufficiently large ball such that for each $s$ with $|s|<1$, $C_s$ is transverse to 
the boundary $\partial B$.
Then, $B\cap C_s$ 
gives raise to a deformation in the sense of Section~\ref{Sub1}. The central fiber is $C_0$, a homogeneous curve, while
a non-central is isomorphic to the intersection of $C$ with a large ball. We can apply Theorem~\ref{mainthe} in this context, noting
that the $M$ number of the singularity $(t^p,t^q)$ is equal to $p+q-\frac{p}{q}-\frac{q}{p}-1$ (see \eqref{eq:Mpq}).
\end{proof}
We remark that the estimate in Proposition~\ref{BMYlike} is very similar to Theorem~4.25 in \cite{BZ}. That result, however, relies on
very difficult BMY inequality.
\begin{ack}
The author wishes to thank M.~Koras, A.~N\'emethi, P.~Russell, A.~Sathaye and H.~\Zol{} for
fruitful discussions on the subject.
\end{ack}

\end{document}